\documentclass[12pt,reqno]{amsart}
\usepackage{amssymb}
\usepackage{amscd}

\newcommand{\nc}{\newcommand}
\nc{\cal}{\mathcal} 

\newtheorem{theorem}{Theorem}[section]

\newtheorem*{theorem*}{Theorem}
\theoremstyle{definition}
\newtheorem{definition}[theorem]{Definition}

\newtheorem{proposition}{Proposition}[section]
\theoremstyle{remark}

\newtheorem{cor}[theorem]{\bf Corollary}

\newtheorem*{claim*}{\bf Claim}
\numberwithin{equation}{subsection}

\newtheorem*{conjectures*}{\rm {\bf Conjectures}}

\renewcommand{\rho}{\varrho}
\renewcommand{\phi}{\varphi}
\renewcommand{\epsilon}{\varepsilon}

\renewcommand{\rho}{\varrho}
\renewcommand{\phi}{\varphi}
\newcommand{\End}{\operatorname{End}}

\newcommand{\supp}{\operatorname{supp}}

\newcommand{\SL}{\operatorname{SL}}

\newcommand{\meas}{\operatorname{\meas}}

\nc{\la}{\langle} \nc{\ra}{\rangle}
\nc{\CA}{\cal A}
\nc{\CBB}{\cal B}
\nc{\CDD}{\cal D}
\nc{\CE}{\cal E}
\nc{\CF}{\cal F} \nc{\CG}{\cal
	G} \nc{\CH}{\cal H} \nc{\CI}{\cal I} \nc{\CJ}{\cal J}
\nc{\CK}{\cal K} \nc{\CL}{\cal L} \nc{\CM}{\cal M} \nc{\CN}{\cal
	N} \nc{\CO}{\cal O} \nc{\CP}{\cal P} \nc{\CQ}{\cal Q}
\nc{\CR}{\cal R} \nc{\CS}{\cal S} \nc{\CT}{\cal T} \nc{\CU}{\cal
	U} \nc{\CV}{\cal V} \nc{\CW}{\cal W} \nc{\CZ}{\cal Z}

\nc{\Ck}{\textsl{k}}

\nc{\fg}{\mathfrak g} \nc{\fii}{\mathfrak i}\nc{\fk}{\mathfrak k}
\nc{\fh}{\mathfrak h} \nc{\fm}{\mathfrak m} \nc{\fn}{\mathfrak n}
\nc{\fA}{\mathfrak A} \nc{\fC}{\mathfrak C} \nc{\fI}{\mathfrak I}
\nc{\fL}{\mathfrak L} \nc{\fS}{\mathfrak S}
\nc{\fz}{\mathfrak z} \nc{\fl}{\mathfrak l}
\nc{\fp}{\mathfrak p}  \nc{\fU}{\mathfrak U}
\nc{\ft}{\mathfrak t}

\nc{\ba}{\mathbb A}
\nc{\bq}{\mathbb Q}
\nc{\br}{\mathbb R}
\nc{\bz}{\mathbb Z}
\nc{\bc}{\mathbb C}
\nc{\bn}{\mathbb N}
\nc{\bg}{\mathbb G}
\nc{\ck}{\mathcal{K}}
\nc{\G}{\Gamma}
\nc{\sm}{\setminus}
\nc{\sub}{\subset}
\nc{\lm}{\lambda}
\nc{\Lm}{\Lambda}
\nc{\al}{\alpha}
\nc{\bt}{\beta}
\nc{\om}{\omega}
\nc{\dl}{\delta}
\nc{\g}{\gamma}
\nc{\Dl}{\Delta}
\nc{\Om}{\Omega}
\nc{\s}{\sigma}
\nc{\ro}{\rho}
\nc{\te}{\theta}
\nc{\SLR}{\operatorname{SL}_2(\br)}
\nc{\GLR}{\operatorname{GL}_2(\br)}
\nc{\PGLR}{\operatorname{PGL}_2(\br)}
\nc{\PSLR}{\operatorname{PSL}_2(\br)}
\nc{\PSLZ}{\operatorname{PSL}_2(\bz)}
\nc{\PGLZ}{\operatorname{PGL}_2(\bz)}
\nc{\SLC}{\operatorname{SL}_2(\bc)}
\nc{\uH}{\mathbb H}
\nc{\fD}{\mathcal{D}}
\nc{\fE}{\mathcal{E}}
\nc{\fO}{\mathcal{O}}
\nc{\haf}{\frac{1}{2}}
\nc{\qtr}{\frac{1}{4}}
\nc{\shaf}{{\scriptstyle\frac{1}{2}}}
\nc{\hlm}{{\scriptstyle\frac{\lambda}{2}}}

\nc{\8}{\infty}
\nc{\7}{{-\infty}}
\nc{\inv}{^{-1}}
\nc{\eps}{\varepsilon}
\nc{\aG}{\mathbf{G}}
\nc{\spn}{\operatorname{Span}}
\nc{\Cm}{\operatorname{CM}}
\nc{\tildl}{\dl^1}
\nc{\chiv}{{\chi_\fp}}
\nc{\psiv}{{\psi_\fp}}
\nc{\piv}{{\pi_\fp}}
\nc{\zt}{Z\setminus T}

\nc{\hra}{\hookrightarrow}

\newcommand{\Hom}{\operatorname{Hom}}

\nc{\ph}{{\rm h}}
\nc{\fgh}{\fg\hat\ }
\nc{\Oph}{{\rm Op}_{\ph}}
\nc{\ras}{\rm a}

\begin{document}

\title{Mass of the Frobenius functional on small balls}


\author{Andre Reznikov}
\address{Department of Mathematics, Bar-Ilan University, Ramat-Gan 52900, Israel}

\email{reznikov@math.biu.ac.il}
\thanks{The research  was partially supported by the ISF grant 1400/19}

\date{\today}



\begin{abstract}
We show that for a co-compact  lattice in a semi-simple quasi-split real group and a fixed abstractly generic automorphic  representation,  Frobenius functional assigns small mass to small balls in the Orbit Method picture of P. Nelson and A. Venkatesh (\cite{NV}; Acta Math., 226, (2021), 1-209). 
\end{abstract}

\maketitle



\section{Introduction}

\subsection{Estimates of Automorphic Functions} 

Automorphic functions are central objects in modern Number Theory, and consequently, bounds on automorphic functions play an important role in modern Analytic Number Theory. These bounds are crucial for the analytic study of periods of automorphic functions and their corresponding $L$-functions. The renewed interest in bounding automorphic functions can be largely attributed to the celebrated work of H. Iwaniec and P. Sarnak \cite{IS}. This paper aims to explore certain aspects related to the small-scale distribution of kernels of automorphic operators. These kernels naturally arise in the theory of automorphic functions as a generalization of a single automorphic function and  play a central role in the groundbreaking work of P. Nelson and A. Venkatesh \cite{NV}, which has greatly influenced our approach.

Since the foundational works of I. Gelfand, S. Fomin, M. Graev, and I. Piatetski-Shapiro, the modern theory of automorphic functions freely employs the language and methods of representation theory. Today, representation theory is also a standard tool in the analytic study of automorphic functions, their periods, and corresponding $L$-functions. One of the most significant recent developments in this direction is the introduction in \cite{NV} of the language and methods from A. A. Kirillov's celebrated Orbit Method \cite{K}.

\subsubsection{Operator Calculus}\label{op-calc-intro} 

The essence of \cite{NV} lies in the construction and application of a version of {\it quantitative} operator calculus for irreducible tempered representations of a real reductive group $G$ ($G$ is assumed to be an unimodular Lie group as in \cite{NV}). We review some notations and basic facts about the construction below and in Section \ref{Oph-sect}, though full details can be found in \cite{NV}.

Let $\mathfrak{g}$ denote the Lie algebra of $G$, and let $i\mathfrak{g}^*:=\mathrm{Hom}_{\mathbb{R}}(\mathfrak{g},i\mathbb{R})$, which we identify with $\mathfrak{g}\hat{}:=\mathrm{Hom}(\mathfrak{g},\mathbb{C}^{(1)})$ via $\xi\mapsto[x\mapsto e^{x\xi}]$.

Choose a sufficiently small {\it fixed} neighborhood $\mathcal{G}\subset \mathfrak{g}$ of $0\in \mathfrak{g}$. In particular, we require that the map $\exp:\mathcal{G}\to G$ is an isomorphism onto its image. We choose Haar measures $dg$ on $G$ and $dx$ on $\mathfrak{g}$ such that $dg=j(x)dx$, where $j:\mathcal{G}\to \mathbb{R}_{>0}$ with $j(0)=1$. Additionally, let $\chi:\mathcal{G}\to \mathbb{R}$ be a smooth characteristic function of a neighborhood of $0$.

We denote by $\mathcal{S}(\cdot)$ the Schwartz space of functions on an appropriate space. Let $d\xi$ be the Haar measure on $\mathfrak{g}\hat{}$ such that Fourier transforms $\mathcal{S}(\mathfrak{g})\ni\varphi\mapsto \varphi\hat{}\in \mathcal{S}(\mathfrak{g}\hat{})$ and $\mathcal{S}(\mathfrak{g}\hat{})\ni a\mapsto a^\vee\in \mathcal{S}(\mathfrak{g})$ are mutually inverse. Here, Fourier transforms are defined by $\varphi\hat{}(\xi)=\int_{\mathfrak{g}} \varphi(x)e^{x\xi}dx$ and $a^\vee(x)=\int_{\mathfrak{g}\hat{}}a(\xi)e^{-x\xi}d\xi$.

We introduce a real parameter $\ph\in(0,1]$ (the ``Planck constant"). The theory presented in \cite{NV} is asymptotic in $\ph\to 0$ with {\it effective} remainders.

Let $\pi$ be a tempered irreducible unitary representation of a real reductive group $G$. The representation $(\pi, V_\pi,\langle\cdot,\cdot\rangle_\pi)$ is equipped with an invariant Hermitian form, which we also view as a $\mathbb{C}$-linear pairing $\langle\cdot,\cdot\rangle_\pi:\pi\otimes\bar{\pi}\to\mathbb{C}$ between $\pi$ and its complex conjugate representation $\bar{\pi}$ (where $(\bar{\pi}, V_{\bar{\pi}})$ denotes the representation in the same space, but with $\mathbb{C}$-multiplication conjugated).

For a ``symbol" $a\in \mathcal{S}(\mathfrak{g}\hat{})$, we define a $\ph$-dependent operator acting on the space of $\pi$ by  
\begin{equation*}\label{Oph}
	\Oph(a)=\Oph(a,\chi:\pi):=\int_{\mathfrak{g}}\chi(\ph x)a^\vee(x)\pi(\exp(\ph x))\ dx\in \mathrm{End}(\pi)\ .
\end{equation*} 
For a symbol $a$ as above, $\Oph(a)$ is a compact trace-class operator. In particular, \cite{NV} obtains the following version of the Kirillov character formula:
\begin{equation}\label{-tr-Oph} 
	\mathrm{tr}(\Oph(a))=\ph^{-d}\left(\int_{\ph\mathcal{O}_\pi}a\ d\omega_{\ph\mathcal{O}_\pi}+O(\ph)\right)\ , 
\end{equation}
as $\ph\to 0$. Here, $\mathcal{O}_\pi\subset \mathfrak{g}\hat{}$ is the coadjoint (multi)orbit assigned to $\pi$ via the Orbit Method, with $2d=2d(\mathcal{O})=\dim \mathcal{O}_\pi$.

It is crucial that \cite{NV} extends the operator calculus beyond $\mathcal{S}(\mathfrak{g})$ to polynomially growing functions, functions depending on $\ph$, and particularly to functions with shrinking support. Moreover, the representation $\pi$ in question can also be $\ph$-dependent (this situation is beyond the scope of our paper). We summarize a few basic facts and notations about $\Oph$ association in Section \ref{Oph-sect}.

For convenience, we also employ the language of Hermitian forms. We denote by $Q^\ph_a$ the form on $V_\pi$ given by
\begin{equation}\label{Qa-forms-intro}
	Q^\ph_a(v)=\langle \Oph(a)v,v\rangle_\pi \ .
\end{equation}
For a real-valued symbol $a\geq 0$, the operator $\Oph(a)$ is asymptotically self-adjoint and non-negative, and the form $Q^\ph_a$ is asymptotically Hermitian. We emphasize that all these statements are effective in $\ph$.

\subsubsection{Automorphic representations and Frobenius functional} Let $\G\subset G$ be a lattice. We will assume in this  paper that $\G$ is {\it co-compact} and denote by $X=\G\setminus G$ the automorphic space. We normalize the $G$-invariant measure $d_X$ on $X$ so that $vol_{d_X}(X)=1$.

 A discrete ($\G$-)automorphic representation is an abstract irreducible unitary representation $(\pi, L_\pi)$ of $G$ together with an isometry $\nu:L_\pi\to L^2(X,d_X)$. Hence a (discrete) automorphic representation is a pair $(\pi,\nu)$. It is well known that $\nu\in \Hom_G(L_\pi, L^2(X))$ preserves the dense space $V_\pi=L^\8_\pi\subset L_\pi$ of smooth vectors, i.e., $\nu:V_\pi\to C^\8(X)$. 

Denote by $\bar e=\G e\in \G\sm G=X$ the image of the identity in $X$. Evaluation at $\bar e$ defines a $\G$-invariant Frobenius functional $I_\nu\in \Hom_\G(V_\pi, \bc)$ given by $I_\nu(v)=\nu(v)(\bar e)$ for $v\in V_\pi$. The Frobenius reciprocity of I. Gelfand, S. Fomin, M. Naimark and I.~Piatetski-Shapiro (\cite{GF}, \cite{GGPS}, \cite{GPS}; see also G. Olshanski \cite{Ol}) then states that this gives the isomorphism $\Hom_G(L_\pi, L^2(X))\simeq \Hom_\G(V_\pi,\bc)$. Given $I$ we can recover $\nu$ by 
$\nu(v)(g)=I(\pi(g)v)$ for $v\in V_\pi$.  We will mostly work with the corresponding rank one Hermitian form $Q_I(v):=|I(v)|^2$ which we call the Frobenius form (associated to $\nu$). In \cite{BR1}, \cite{BR2}, we determined the Sobolev $L^2$-class of the Frobenius functional and this will play a crucial role in what follows (see also \cite{Ot} and \cite{S} for the H\"older  class of $I$, and \cite{BR2} for the Besov class, although all these results are valid only for $G=\SLR$, while the main result of \cite{BR2} is general).

\subsection{Automorphic functions and kernels of automorphic operators} From the point of view of Frobenius reciprocity, the value of an automorphic function $\phi_v$ at a point (e.g., at $\bar e$)  is given by the pairing $|\phi_v(\bar e)|^2=|I(v)|^2=\langle Q_I,Q_v\rangle $. Here we denote by $\phi_v:=\nu(v)$ the automorphic function corresponding to a vector $v\in V_\pi$, $Q_v(\cdot)=|\langle \cdot, v\rangle|^2$ is the corresponding rank one Hermitian form and $Q_I$ is the Frobenius form. Unfortunately, it is very difficult to work with a single vector in the Orbit Method picture of \cite{NV} since it corresponds to going to the ``Planck scale" (i.e., to vectors microlocalized to $\ph^\haf$ scale). What Nelson and Venkatesh have shown in \cite{NV} is that for many questions (e.g., bounds on periods of automorphic functions) one can equally well work with a collection of vectors $\sum v_i\otimes \bar v_i\in \pi\otimes \bar \pi\simeq End(\pi)$, i.e., with operators. Hence instead of rank one forms $Q_v$ one is naturally led to more general forms $Q_a^\ph$. 

Our main object in this paper will be the quantity 
\begin{equation}\label{QI-Qa}
	E_\ph(a):= \langle Q_I,Q_a^\ph\rangle \ .	
\end{equation}
Here we view $Q_I=I\otimes \bar I\in \pi^*\otimes \bar\pi^*$, $Q_a^\ph=\sum v_i\otimes u_i\in \pi\otimes\bar\pi$ and $\langle \cdot,\cdot \rangle$ comes from the natural pairing $\langle \cdot,\cdot\rangle_\pi: \pi^*\otimes\bar\pi\to \bc$ associated with the unitary structure on $\pi$. Note that for an automorphic representation $(\pi,\nu)$, its complex conjugate $\bar\pi$ is also automorphic with respect to $\bar\nu$.  It is easy to see that for $a\in\cal{S}(\fg)$, the pairing $ \langle Q_I,Q_a^\ph\rangle $ is indeed finite (e.g., by appealing to the Sobolev class of $I$; see Section \ref{Frob-sect}).  The quantity $E_h(a)$ is the central object of \cite{NV} in applying the Orbit  Method to periods of automorphic representations (see \cite{NV},  \S 26.3 ). 

Since $Q_I$ is a rank one form (and this will be essential in what follows), one can have the following  less symmetric expression
\begin{equation}\label{QI-Opa-0}
	E_\ph(a):= \langle I,  \Oph(a)(I)\rangle_\pi=I ( \Oph(a)(I))\ .	
\end{equation}
It is easy to see that  $\Oph(a)(I)\in V_\pi$ is a smooth vector and hence the above evaluation is well-defined. In \cite{NV} it is shown that for $a=(\ras)^2$ with ${\rm a}\geq 0$, $\Oph(a)$ is asymptotic to $\Oph(\ras)^2,$ and we can work with a more symmetric expression $||\Oph(\ras)(I)||^2_\pi$. 

More generally, we can consider the functional $I_x(v)=\nu(v)(x)$ on $\pi$ coming from an evaluation at a general  point $x\in X$, and the corresponding form $Q_{I_x}$. For a co-compact $\G$, this is however usually not necessary since   $\langle Q_{I_x},Q_a^\ph\rangle=\langle Q_{I},\pi(g)Q_a^\ph\bar\pi(g\inv)\rangle$ is asymptotic as $\ph\to 0$ to $E_\ph(g\cdot a)=\langle Q_{I},Q_{g\cdot a}^\ph\rangle$ for $g\in G$ such that $x=\bar eg$. Here $g\cdot a$ is the usual action on functions on $\fgh$. For non-uniform $\G$, the behavior of  $\langle Q_{I_x},Q_a^\ph\rangle$ depends on whether $\pi$ is cuspidal and will not be covered in this paper (see \cite{BR3}, \S3.3 for a discussion). 

\subsubsection{Mass distribution of Frobenius functional} We can view $E_\ph:\CS(\fg)\to\bc,\ a\mapsto E_\ph(a)$ as an $\ph$-dependent distribution $E_\ph$ on symbols.  We are interested in studying the distribution of mass of $E_\ph$. According to the Orbit Method principle, this is a distribution which essentially is supported on the orbit $\ph\CO_\pi$. As $\ph\CO_\pi\to\CN$ for $\ph\to 0$, we obtain a distribution on the nilcone  $\CN\subset \fgh$.  The general expectation is that in a weak sense this distribution is the invariant integral on $\CN$.  Below we state theorems supporting this point of view. 

{\bf Assumptions.}\label{assumption} Asymptotic operator calculus of \cite{NV} is largely expressed in terms of the limit orbit $\lim_{\ph\to 0}\ph\CO_\pi=\CO\subset \fgh$. In order for this object to be non-empty, one needs to impose certain conditions. The  real group $G$ is assumed to be quasi-split. For a quasi-split real group $G$  and a fixed representation  $\pi$, the condition that  $\CO\not=\emptyset$ is equivalent to $\pi$ being generic (i.e., admitting  a non-degenerate Whittaker model as an abstract representation; \cite{NV}, \S11.4). From now on we will always assume that, as an abstract representation, $\pi$ is {\it generic}. 

Coarse geometric methods of \cite{BR3}, translated and enhanced by the language of \cite{NV} imply the following 

\begin{theorem}[\cite{BR3}]\label{BR-thm} Let $\pi$ be a fixed cuspidal tempered representation (in particular $\G$ is not assumed to be uniform) and assume that $\pi$ is generic as an abstract representation.  For a fixed non-negative symbol $a\in\CS(\fgh)$ and for any $g\in G$, there exists a constant $C=C(\pi,a,g)>0$ such that 
	\begin{equation}\label{BR-thm-eq}
		E_\ph(g\cdot a)\leq C\cdot \ph^{-d}\int_{\ph\cal{O}_\pi}g\cdot a\ d\om_{\ph\cal{O}_\pi}\ ,
	\end{equation}
	as $0<\ph<1$.
\end{theorem}
The dependence of $C(\pi, a, g)$ is effective in all parameters and in particular depends on the injectivity radius of the point $\bar eg$ in $X$.  The quantity on the right in \eqref{BR-thm-eq} is asymptotic to $tr(\Oph(g\cdot a))$. The proof of the above theorem is based on ``fattening/thickening  the cycle" technique (see \cite{BR3}, \S 5.3 and Appendix \ref{fatt} below).

In \cite{NV}, Nelson and Venkatesh significantly improved the above bound. Namely,  they proved the following 

\begin{theorem}[\cite{NV}]\label{NV-thm} Assume that $\G$ is co-compact and $\pi$ as in the previous theorem. For a fixed non-negative symbol $a\in\CS(\fg)$ and for any $g\in G$, the following limit holds
	\begin{equation}\label{NV-thm-eq}
		\ph^{d}\cdot E_\ph(g\cdot a) \xrightarrow[\ph\to 0]{\text{}} \int_{\ph\cal{O}_\pi}g\cdot a\ d\om_{\ph\cal{O}_\pi}\ .
	\end{equation}
\end{theorem}
The rate of convergence in the theorem is not effective both in $\ph$, $a$ or $g$. The proof in \cite{NV} is based on discovering a remarkable unipotent hidden symmetry in the problem coming from the geometry of co-adjoint orbits (namely, that  $\ph\CO_\pi\to\CN$). This allowed \cite{NV} to use Ratner's theorem for unipotent flows. We sketch in Appendix \ref{repack}  a slight variation of their proof.  

Symbols of the type covered by the above theorems have essential support in  a {\it fixed} ball $B\subset\fgh$, and are usually called {\it dyadic symbols} (due to the property of having spectral decomposition spread over dyadic sets, e.g., as in Theorem 1.1 of \cite{NV}). The non-effectiveness of the result stems from non-effectiveness of Ratner's theorem and should be remedied by an expected effective version.  In particular, an {\it effective} version of Ratner's theorem should allow one to deal with shrinking with $\ph\to 0$ symbols.  Our main result shows that   this is indeed  possible to achieve without invoking (an effective) Ratner's theorem.

\begin{theorem}\label{main-thm} Assume that $\G$ is co-compact and $\pi$ as before.  Let $\al\in C^\8_c(\fgh)$ be a fixed non-negative function with the support  $\supp(\al)\subset B_1(\fgh)$ in a unit ball. For $0<\dl<1/2$ and $0\not=\xi_0\in\CN_{reg}$, denote by 
	\begin{equation}\label{a-symbol-thm}
a_\ph(\xi)=a(\ph,\xi_0,\dl,\al)(\xi)=\al(\xi_0-\xi/\ph^\dl)\in C^\8_c(\fgh), \ \xi\in\fgh\ . 	\end{equation} 
There exists $\s=\s(\dl)>0$ and a constant $C''=C''(\pi,\al, \dl,\xi_0,\s)>0$ such that the following bound holds 
	\begin{equation}\label{thm-eq}
		E_\ph( a_\ph)\leq C'' \cdot \ph^{-d+\s}\ .
	\end{equation}
\end{theorem}
 The constant $C''(\pi,\al, \dl,\xi_0,\s)$ is effective in all parameters and  in particular, for a fixed  $\xi_0\not=0$, depends on the {\it diameter} of $\supp(a_\ph)\asymp \ph^\dl$. Symbols $a(\ph,\xi_0,\dl,\al)$ were introduced and studied in \cite{NV}. Note that $tr(\Oph(a_\ph))$ is of order of $\ph^{-(1-\dl)d}$. Our argument does not recover such a rate. The main thrust of the above theorem is that on symbols with sub-dyadic support the mass of $Q_I$ is polynomially in $\ph$ smaller than trivial bound of the order of $\ph^{-d}$ which is valid for dyadic symbols. In fact our proof works with symbols which are skewed only in certain directions (those expanding under the action of a split torus) and in complementary directions stay dyadic. 
 
 The idea of the proof of the above theorem is based on an interplay between two norms bounding the Frobenius functional. One is the sup norm on $X$ and another one is (a variation of) the $L^2$ Sobolev norm. By switching back and forth, we use the fact that sup norm is invariant while one can {\it decrease} the Sobolev norm of a vector by ``pushing it into the sink" with the action of $G$  (i.e., moving the micro-localized vector closer to the origin  of the nilcone $\CN\subset \fgh $ in the Orbit Method  picture  of \cite{NV}). 
 We note that a similar idea was employed in \cite{BR0}, \S3.4, but a clear picture  emerges only thanks to the ``geometrization" of \cite{NV}.
 
 We also note that the above result has an effect on the sup norm problem for automorphic functions. Namely, the square of any automorphic function in the image of $\Oph(a_\ph)$ is bounded by $E_\ph(a_\ph)$. The ``trivial" bound comes from the quantitative Frobenius theorem (\cite{BR2}; also see Section \ref{Frob-class}) and is of order of $\ph^{-d}$.  Theorem \ref{main-thm} gives a polynomial improvement, but of course only for vectors in ${\rm Im}(\Oph(a_h))$ (i.e., localized to a small micro-support. Finite $K$-types vectors in $\pi$ are not of this type; see discussion in \cite{NV}, \S1.7).
Note that for a dyadic symbol $a$, the space ${\rm Im}(\Oph(a))$ is essentially $\ph^{-d}$-dimensional and hence there could not be improvement of sup norm on such a subspace. 

\subsubsection{Conjectures}\label{conj} In the spirit of QUE and the above theorems, it is natural to formulate the following set of conjectures.

\begin{conjectures*} Consider $\ph$-dependent  cuspidal representations $(\pi_\ph,\nu_\ph)$ (as in \cite{NV}). We assume that the infinitesimal character $\lm_\pi$ of $\pi_\ph$ satisfies the bound $|\lm_\pi|\leq\ph\inv$, and  $\pi_h$ admits a limit orbit $\CO=\lim\limits_{\ph\to 0} \ph\CO_{\pi_h}$ (see \cite{NV}, \S 11.4). Let  $a_\ph=a(\ph,\xi_0,\dl,\al)$ be as in Theorem \ref{main-thm}. 
	\begin{enumerate}
	
			\item For $\pi_\ph=\pi$ fixed, there exists $1/2>\dl^1_0>0$ such that for all $\dl<\dl^1_0$, there exist  $C=C(\pi,\al,\xi_0)$ such that 
		\begin{equation}\label{conj2-eq}
			E_\ph( a_\ph)\leq C \cdot \ph^{-d}\int_{\ph\cal{O}_\pi} a_\ph\ d\om_{\ph\cal{O}_\pi}\ ,\ {\rm for\ all} \  1>\ph> 0. 
		\end{equation}
	
		\item For $\pi_\ph=\pi$ fixed,  there are $1/2>\dl^2_0>0$ and $\s'>0$ such that for all $\dl<\dl^2_0$
	\begin{equation}\label{conj1-eq}
		E_\ph( a_\ph)=  \ph^{-d}\int_{\ph\cal{O}_\pi} a_\ph\ d\om_{\ph\cal{O}_\pi}+O(\ph^{(-1+\dl)d+\s'})\ ,\ as\  \ph\to 0. 
	\end{equation}

			\item For $\pi_\ph$, $\lm_{\pi_\ph}\to\8$ and $\xi_0\in \CO=\lim\limits_{\ph\to 0} \ph\CO_{\pi_h}$ there exists $1/2>\dl^3_0>0$ such that for all $\dl<\dl^3_0$, there exists $\s''=\s''(\pi,\al,\xi_0)>0$ such that 
		\begin{equation}\label{conj3-eq}
			E_\ph( a_\ph)\ll \ph^{-d+\s''},\ {\rm for\ all}\  1>\ph> 0. 
		\end{equation}
		
			\item For $\pi_\ph$ satisfying the same conditions as in (3) and $\xi_0\in\CO$, there exists $1/2>\dl^4_0>0$ such that for all $\dl<\dl^4_0$, there exist a constant $C'=C'(\CO,\al,\xi_0)$ such that 
		\begin{equation}\label{conj4-eq}
			E_\ph( a_\ph)\leq C' \cdot \ph^{-d}\int_{\ph\cal{O}_\pi} a_\ph\ d\om_{\ph\cal{O}_\pi}\ ,\ for\ all \  1>\ph> 0. 
		\end{equation}

	\end{enumerate}
	
\end{conjectures*}
Conjectures (1),(2) and (4) imply a variety of subconvexity bounds for a number of periods. The value of $\dl_0^i$'s is a delicate issue even in the conjecture (1) (e.g., does the value $\dl_0^1=1/2$ always hold). In conjecture (4) even for $\SL_2(\bc)$, we have $\dl_0^3\leq 1/4$ as the example of Z. Rudnick and P. Sarnak \cite{RS} shows. An effective version of Ratner's theorem would imply conjecture (1) and (2) with some $\dl_0^{1,2}>0$. Conjecture (4) (or a strong enough bound in Conjecture (3)) would imply resolution of a variety of QUE-type problems.  

\subsubsection{Remark: $\SLR$.}  For the group $G=\SLR$ and a co-compact lattice $\G$ without a torsion, the automorphic space $X$ corresponds to the unit cotangent bundle over  a compact hyperbolic Riemann surface. In this case, the classical microlocal analysis could be translated into notions of representation theory and in particular eigenfunctions of Laplace-Beltrami operators became associated to automorphic  representations (see \cite{GGPS}, \cite{Z}). Under such an interpretation, the quantity $Q_I$ essentially becomes what is known as Husimi function in the microlocal analysis. Effective Ratner theorem is known for $\SLR$, and hence bounds in conjectures (1) and (2)   hold (for some $\dl_0^{1,2}>0$). This type of questions about distribution of the Frobenius functional in the ``vertical" direction (as opposed to the ``horizontal" when the infinitesimal character of $\pi$ grow) was researched in a number of papers. These results are mostly expressed  in term of the behavior of automorphic $K$-types in a fixed or a slowly changing representation.  In particular, one can show that $K$-types in a fixed representation are effectively equidistributed (\cite{R}) and $L^4$-norms of $K$-types are uniformly bounded (\cite{BR4}). Also S. Zelditch (\cite{Z1}) showed equidistribution on average over a ``short interval" of $K$-types. Non-trivial results on sup-norm problem for $K$-types of $SL_2(\bc)$, were recently obtained in \cite{BHMM}. $\Oph$ calculus allows one to ask (and sometime answer) more flexible questions, although single $K$-types do not fit well into $\Oph$ calculus (see a discussion in \cite{NV}, \S1.7) .   

\section{Frobenius functional, norms and operator calculus} 

\subsection{Frobenius functional}\label{Frob-sect} Here we review  results on the norm of Frobenius functional obtained  in \cite{BR2} .

Let $(\pi,\nu)$ be an automorphic cuspidal representation of $G$. We fix an invariant\linebreak Hermitian form $P$ and denote by $L=L_\pi$ the corresponding  Hilbert space. We assume that $\nu: L\to L^2(X)$ is an isometry. Let $V=L^\8$ be the space of smooth vectors of $\pi$. It is well-known that $\nu:V\to C^\8(X)$ and the Frobenius reciprocity of Gelfand--Graev--Naimark--Piatetski-Shapiro provides the isomorphism  $\Hom_G(L_\pi, L^2(X))\simeq \Hom_\G(V_\pi,\bc)$. Let $N$ be a {\it uniformly} continuous Hermitian norm on representation $\pi$ and $Q$ the corresponding\linebreak Hermitian form. Continuity of $N$ means that we have a well-defined continuous ``distortion"  function $d(g)=d_Q(g):=||\pi(g)||^2_N$ on $G$ (i.e., $Q(\pi(g)v)\leq d_Q(g)Q(v)$ for all $v\in V$).  We are interested in the $N$-norm of the Frobenius functional $||I||_N:=\sup_{v\in V}|I(v)|/N(v)$. A possible answer is given in terms of the relative trace of forms $P$ and $Q$. In \cite{BR2}, we constructed the relative trace of Hermitian forms $tr(P|Q)\in \br_+\cup\8$ for  two non-negative Hermitian forms $P$ and $Q$ with $Q$ positive definite. For a form $Q\geq c\cdot P$ with $c>0$, the form $P$ is given by a bounded self-adjoint operator $A$ with respect to the form $Q$, and we have $tr(P|Q)=tr(A)$. 

We have the following very general theorem on the (Hermitian) class of the Frobenius functional. 

\begin{theorem*}{\cite{BR2}} The following estimate holds \begin{equation}\label{Frob-class}
		||I||_N^2\leq C\cdot tr(P|Q)\ ,
	\end{equation} where $C=C(\pi,\nu, X)>0$ is an effectively computable constant. 
\end{theorem*}
For a co-compact $\G$, this result is sharp (i.e., for a co-compact $\G$ one has an analogous lower bound; see \cite{BR3}). 

Note that for  the invariant  form $P$, a form $Q\geq cP$ with $c>0$,  and the operator $A$ representing $P$ with respect to $Q$, we have $tr(Q_I|Q)=||I||^2_N=P(A(I),I)=\bar I(A(I))$. According to the above theorem, all the above quantities are finite if $tr(P,Q)=tr(A)$ is finite {\it and} the form $Q$ is {\it uniformly} continuous. If we assume that the form  $P$ is also positive definite (e.g., $P$ is the invariant Hermitian form on $V_\pi$), then we have $Q(v,u)=P(B(v),u)$ for some self-adjoint positive operator and $A=B\inv$. The operator $B$ is usually non-bounded.    Also note that for a non-negative self-adjoint (w.r.t. $P$) operator $D$ such that $D^2=A$, we have  $||I||^2_N=P(D(I),D(I))$.

Proof of the above theorem is based on the basic relation  $P=\int_Xg_xQ\ dx$, where $xg_x=\bar e$ and hence  $tr(P|Q)=\int_Xtr(Q_{I_x}|Q)dx=\int_X||I_x||^2_Ndx$, $I_x=\pi(g_x)I$, and on the {\it uniform} continuity of the norm $N$ (i.e., the existence of continuous distortion function $d_Q(x)$; see \cite{BR2}, \S 3.2). 
The essence of  Conjectures (1), (2) and (4) is that a similar bound could be extended to certain forms which are {\it not} uniformly continuous. In particular, forms $Q^\ph_a$ introduced in Section \ref{op-calc-intro} are far from being uniformly continuous. For a compactly supported symbol $a$, the operator $\Oph(a)$ is essentially a projector onto a finite dimensional subspace in $V_\pi$.

\subsection{Continuous Hermitian forms} In order to apply the above theorem, we need to exhibit a continuous Hermitian form on $\pi$. A very natural construction of such forms comes from classical $L^2$ Sobolev norms. We refer to \cite{BR2}, Appendix B and \cite[\S 3.2]{NV},  for a discussion. 

Standard Sobolev norms could be introduced via the use of an ``elliptic" Casimir operator $\Dl=1-\sum_{x\in\CBB}x^2\in\fU(\fg)$ in the universal enveloping algebra (here the summation is over a basis of $\fg$ which we assume to be orthogonal). The $L^2$ Sobolev norm with an integer index $\s\in\bz$ is defined via the Hermitian form $H_\s(v,u):=\langle \Dl^\s v,u\rangle_P$. The standard interpolation technique then extends this family of Hermitian forms to a real parameter $\s\in\br$. Note that this extension is relevant if one is interested in the exact class of the Frobenius functional (e.g., for $\SLR$, the $\s$-Sobolev norm of $I$ is finite for any $\s> 1/2$). In practice, it is more convenient to work with a family of norms $H_{\s,r}=r^2P+H_\s$ with $r\geq 0$ (see \cite{BR2}, \S4.1.1). As noted in \cite{B}, one expects that the structure of Sobolev norms is richer in higher rank. We hope to return to this subject elsewhere. 

We now discuss a  more geometric way to think about continuous Hermitian forms on representations which is provided by the Orbit Method of P. Nelson and A. Venkatesh. 

Let $b=b_\ph$ be a possibly $\ph$-dependent real non-negative symbol on $\fgh$. We consider the following form $H_b^\ph$ on $V_\pi$ (we will assume that $\pi$ is fixed) defined by  
\begin{equation}\label{key}
	H_b^\ph(v,u)=\langle \Oph(b)v,u\rangle_P\ , 
\end{equation} for $v,u\in V_\pi$. We purposely want to distinguish between forms $	H_b^\ph$ which we intend to view as {\it uniformly continuous } Hermitian forms and forms $Q_a^\ph$ introduced in \eqref{Qa-forms-intro} which are essentially supported on some finite-dimensional subspace of $V_\pi$ and are not uniformly continuous. For example, we have ${\rm Op}_1(p)=\Dl$ for $p(\xi)=1+|\xi|^2$ (see \cite{NV}, \S 5.2) and more generally, for $\s\in\bz$, $Op_1(p^\s)=\Dl^\s$. These operators define Sobolev type Hermitian forms $H_\s(v,u)=H^1_{\s,0}(v,u)=P({\rm Op}_1(p^\s)v,u)$ (see  \cite{NV}, \S3.2). In the spirit of the Kirillov formula we  expect that 
\begin{equation*}
||I||_{N_{\s}}^2\leq C\cdot tr(P|H_\s)=C\cdot tr(\Dl^{-\s})\asymp\int_{\CO_\pi} p^{-\s}(\xi)d\om_\CO\ .
\end{equation*} It is easy to see that the last integral is finite for $-\s<-\haf d$ (see \cite{NV}, \S11.3; here as before $d=d(\CO_\pi)=\haf\dim \CO_\pi$). 
\subsubsection{Forms $H_{s,r}^\ph$} Tweaking the non-standard  notation of \cite{NV}, we denote by $\langle \xi\rangle_r= (r^2+|\xi|^2)^\haf$ for $r>0$.  We consider a symbol $b^{s,r}(\xi):=\langle \xi\rangle_r^{s}$ for  real parameters $s$ and $r> 0$, and the corresponding Hermitian forms $H_{s,r}^\ph(v,u):=\langle \Oph(b_h^{s,r})v,u\rangle_P$. The  relative trace $tr(P|H_{s,r}^1)$ is finite for any $-s<-d$.  The form $H^1_{s,r}$ is {\it uniformly continuous} due to the fact that the function $\langle \xi\rangle_r$ is self-similar under the co-adjoint action of $G$. Hence  we obtain  the following explication of  \eqref{Frob-class},
\begin{equation}\label{Frob-class-Sob}
	||I||_{N^1_{s,r}}^2\leq C_{s,r}\cdot tr(P|H^1_{s,r})\ ,
\end{equation} for any $-s<-d$. Here $C=C(\pi,\nu, X,s)>0$ is an effectively computable constant. Denoting $a^{s,r}=(b^{s,r})\inv$, we expect that $tr(P|H^1_{s,r})\asymp tr({\rm Op}_1(a^{s,r}))\asymp\int_{\CO_\pi} a^{s,r}(\xi)d\om_\CO\ $. The last quantity is purely geometric  and is easy to control even for $s\to-d$. 

We would like to re-write Hermitian forms $H^1_{s,r}$ (or some equivalent forms) in terms of $\Oph$. The reason is that we want to compare/compose these forms with forms defined by symbols $a(\ph,\xi_0,\dl,\al)$  introduced in \eqref{a-symbol-thm}, Theorem \ref{main-thm}, but forms $H^1_{s,r}$ are defined at the scale $\ph=1$ and at a smaller scale their symbol is too singular near $\xi=0$. 
To this end, we define forms  $H_{s,r}^\ph(v,u):=\langle \Oph(b_h^{s,r})v,u\rangle_P$ by an $\ph$-dependent symbol 
\begin{equation}\label{b_h}
	b_\ph^{s,r}(\xi)=\langle \xi/\ph\rangle_r^{s}=(r^2+|\xi/\ph|^2)^{\haf s} \ .
\end{equation} We also consider the inverse symbol $a^{s,r}_\ph=(b^{s,r}_\ph)\inv$. Here $s,\ r\geq 0$ are real parameters. The range of parameters we are interested in is as follows. The order $s$ should satisfy  $s=d+\eps$ for small $\eps>0$ since we require finiteness of the trace $tr(P|H_{s,r}^\ph)<\8$. The parameter $r$ will be set at $r=\ph^{-\kappa}$ for   $1>\kappa>1/2$ to be chosen later.  Such a choice of $r$ is necessary in order for these symbols to  fit into the symbol classes of \cite{NV}, \S4  (for $r=O(1)$, the ``geometry" of $b^{s,r}_\ph$ at $\xi=0$ is on the scale of $\ph\inv$  and this is not allowed for $\Oph$). We have the following 
\begin{proposition} Let $s\in\br_{\geq 0}$, $1>\kappa>1/2$, and denote by	\begin{equation}\label{b-symbol}
		b_\ph^{s,\kappa}(\xi)=\ph^{-s}(\ph^{2(1-\kappa)}+|\xi|^2)^{\haf s}
	\end{equation} and by $a_\ph^{s,r}(\xi)=\left[b_\ph^{s,r}(\xi)\right]\inv$. We have
\begin{enumerate}
		\item  $b_\ph^{s,\kappa}\in \ph^{-s}S^s_{1-\kappa}$ and  $a_\ph^{s,\kappa}\in \ph^{s}S^{-s}_{1-\kappa}$\ .
		\item  $\Oph(b_\ph^{s,\kappa})\in \ph^{-s}\Psi^m_{1-\kappa}$\ , for any $m>s$,  and $\Oph(a_\ph^{s,\kappa})\in \Psi^n_{1-\kappa}$ for any $n>-s$.
		\item  The operator $A(h)=\Oph(b_\ph^{s,\kappa})\circ\Oph(a_\ph^{s,\kappa})$ is an invertible operator in $\End(\pi)$ with the norms of $A(\ph)$ and of $A(\ph)\inv$ bounded by a constant $C>0$ independently of $\pi$ for $0<\ph\leq\ph_0$ and some $\ph_0>0$. 
		
		\item  For $s>d$, the operator $\Oph(a_\ph^{s,\kappa})$ is of a trace class on $\pi$. There exists $C_s>0$ depending on $s$, but not on $\pi$ or $\kappa$ such that $tr(\Oph(a_\ph^{s,\kappa}))\leq C_s$.
		
\end{enumerate}
			\end{proposition}
\begin{proof} Claims in (1) and (2) are  standard  in $\Oph$ calculus of \cite{NV} (see Section \ref{Oph-sect} below). The proof of (3) follows the lines of \cite{NV}, \S12.2, Lemma 12.1, where such a statement is proved for $\haf s=N\in\bz_{\geq 0}$ and for the symbol $\dl_s(\xi)=(1+|\xi|^2) ^{N}$ giving $\Oph(\dl_N)=\Dl_\ph^N$ with $\Dl_\ph=1-\ph^2\sum_{x\in\CBB}x^2$. Bound in (4) follows from \cite{NV}, Lemma 11.5.
	\end{proof}
	Our choice of the symbol $(r^2+|\xi|^2)^{\haf s}$  produces   Hermitian forms $H_{s,r}^\ph$ which are variants of the $L^2$-Sobolev form of index $s$ on $\pi$ (see below). In fact, one can use interpolation of Hermitian forms to extend Sobolev norms to general real index $s$ (see \cite{BR1}, Appendix B). Note that we do not use automorphic realization of  forms $H_{s,r}^\ph$ on spaces of functions on $X$ (e.g., forms $\langle\cdot,\cdot\rangle_{\pi^s}$ below are defined via the operator $\Dl$ which corresponds to an explicit  differential operator acting on functions on $X$). Our connection to the automorphic picture comes only from the Frobenius functional $I$.  

\begin{cor}\label{I-bound-cor}\ 
\begin{enumerate}
		\item  For any $s>d$, $1>\kappa>1/2$ and $r=\ph^{-\kappa}$, there exists a constant $C_s>0$ such that \begin{equation*}
			||I|^2_{H^\ph_{s,r}}\leq C_s\ . 
		\end{equation*}
			\item  Let  $(\pi,\nu)$ be an automorphic cuspidal representation $v\in V_\pi$ and $\phi_v\in C^\8(X)$ be the corresponding automorphic function. For $s$ and $r$ as above, we have \begin{equation*}
				|\phi_v(\bar e)|^2=|I(v)|^2\leq C_s\cdot \langle \Oph(b_\ph^{s,\kappa})(v),v\rangle\ .
			\end{equation*}
				\item  Assume that $\G$ is co-compact. For $s$, $r$ and $\phi_v$ as above, we have \begin{equation}\label{sup-norm-bd}
					\sup_X|\phi_v|^2\leq C_{s,X}\cdot \langle \Oph(b_\ph^{s,\kappa})(v),v\rangle\ .
				\end{equation}
	\end{enumerate}
	
	\end{cor} 

\subsection{Symbol and operator classes}\label{Oph-sect} We copy from \cite{NV}, \S3 and \S4. An $\ph$-dependent symbol $a : \fgh\to\bc$ is a function which depends
– perhaps implicitly – upon $\ph$: $a(\xi) := a(\xi; \ph)$, and denote by $a_\ph(\xi) := a(\ph\xi) = a(\ph\xi; \ph)$
 the rescaled $\ph$-dependent function.

\begin{definition}[\cite{NV}, \S4.4]
Let $m\in\br,\ \dl\in [0; 1)$. Let $a : \fgh\to\bc$ be a smooth $\ph$-dependent function. We
write $a\in  S_\dl^m=S_\dl^m(\fgh)$
if for each multi-index $\al\in\bz_{\geq 0}^n$, $n=\dim(G)$,  there exists $C_\al>0$ so that for all $\xi\in\fgh$ and $\ph\in (0; 1]$, 
\begin{equation}\label{h-symbol-def}
	|\partial^\al a(\xi)|\leq C_\al \ph^{-\dl|\al|} \langle\xi\rangle^{m-|\al|}\ .
\end{equation}
\end{definition}
Informally, $S_\dl^m$  consists of elements which oscillate at the scale $\xi+O(\ph^\dl\langle\xi\rangle)$.

Fix a basis $\CBB := \CBB(\fg)$ of $\fg$, and set
$\Dl  := 1 - \sum_{x\in\CBB}x^2\in\fU$. 
Also denote by $\Dl$ its image $\pi(\Dl)\in \End(\pi)$. The operator $\Dl$ induces a densely-defined self-adjoint positive operator on $\pi$ with bounded inverse
and $\Dl\inv\leq 1$. For $s\in\bz$, define  (Sobolev) inner product $\CS_s=\langle\cdot,\cdot\rangle_{\pi^s}$ on $\pi^\8$ by $\langle v,u\rangle_{\pi^s}:=\langle\Dl^sv,u\rangle $
The inner product on $\pi$ induces a duality between $\pi^s$ and $\pi^{-s}$, $\pi^{\8}:=\cap\pi^s$ and $\pi^{-\8}:=\cup\pi^s$.  An operator on $\pi$ is a linear map $T:\pi^\8\to\pi^{-\8}$. For $x\in\fg$, consider the commutator $\theta_x(T):=[\pi(x),T] $. The map $x\mapsto \theta_x$ extends to an algebra morphism $\fU\to \End(\{{\rm operators\ on \ }\pi\})$
\begin{definition}[\cite{NV}, \S3.3]
	For $m\in\bz$, an operator $T$ on $\pi$ has order $\leq m$ if for each $s\in\bz$ and $u\in\fU$ the operator $	\theta_u(T)$ induces a bounded map
	\begin{equation}\label{op-class-def}
		\theta_u(T):\pi^s\to\pi^{s-m}\ .
	\end{equation}
\end{definition}
The space of operators of order $\leq m$ is denoted by $\Psi^m:=\Psi^m(\pi)$. Denote by $\Psi^{-\8}:=\cap\Psi^m$ the space of smoothing operators and by  $\Psi^{\8}:=\cup\Psi^m$ the space of finite order operators. The ${\rm Op}$ calculus provides a map ${\rm Op}_1(S^m)\subseteq \Psi^m$ for all $m\in\bz\cup\{\pm\8\}$ (\cite{NV}, \S5.6, Theorem 5.6). 

For $\ph$-scaled symbols, one needs a slight tweak of operator classes (see \cite{NV}, \S5.3). 
Let $\pi=\pi_\ph$ be an $\ph$-dependent unitary representation. Fix $\dl\in[0,1)$. Denote
by $\Psi^m_\dl$ the space of $\ph$-dependent operators $T = T(\ph)$ on $\pi$ with the property that
for each $u\in\fU$ and $s\in\bz$, there exists $C_{u,s}\geq 0$  (independent of $\ph$) so that for all $\ph\in (0,1]$,
$||\theta_u^\dl(T)||_{\pi^s\mapsto\pi^{s-m}}\leq C_{u,s}$, where $u\mapsto \theta_u^\dl$ denotes $\ph^\dl$-re-scaled map: for $u=x_1\dots x_n$, $\xi\in\fg$, $\theta_u^\dl=\ph^{n\dl}\theta_u$. The basic result on $\Oph$ association  is then the following 
\begin{theorem*}[\cite{NV},\S5.6] Fix $\dl\in[0,\haf)$. For $m\in\bz$, we have $\Oph(S^m_\dl)\subseteq \ph^{\min(0,m)}\Psi^m_\dl$ .
\end{theorem*} The notation $\ph^a\Psi^m_\dl$ means operators in $\Psi^m_\dl$ with norms bounded by $\ph^a$. 

\section{Proof of Theorem \ref{main-thm}} 

\subsection{} Let $(\pi,\nu)$ be a non-trivial tempered generic representation of $G$ for a {\it co-compact} lattice $\G$, and $I: V_\pi\to\bc$ the corresponding Frobenius functional. We denote by $Q_I(\cdot)=|I(\cdot)|^2$ the corresponding rank one Hermitian form on $\pi$. 

Let $\al\in C^\8_c(\fgh)$ be a fixed non-negative function with the support  $\supp(\al)\subset B_1(\fgh)$ in a unit ball. For $0<\dl<1/2$ and $0\not=\xi_0\in\CN_{reg}$, denote by 
\begin{equation}\label{a-symbol-thm-2}
	a_\ph(\xi)=a(\ph,\xi_0,\dl,\al)(\xi)=\al(\xi_0-\xi/\ph^\dl)\in C^\8_c(\fgh), \ \xi\in\fgh\ . 	\end{equation} 
We define 
\begin{equation}\label{QI-Qa-2}
	E_\ph(a_\ph):= \langle Q_I,Q_{a_\ph}^\ph\rangle \ .	
\end{equation}
Here we view $Q_I=I\otimes \bar I\in \pi^*\otimes \bar\pi^*$, $Q_{a_\ph}^\ph(\cdot,\cdot)=\langle\Oph(a_\ph)(\cdot),\cdot\rangle$, $Q_a^\ph=\sum v_i\otimes u_i\in \pi\otimes\bar\pi$ and $\langle \cdot,\cdot \rangle$ comes from the natural invariant pairing $\langle \cdot,\cdot\rangle_\pi$ on $\pi$. The pairing $	\langle Q_I,Q_{a_\ph}^\ph\rangle $ is well-defined since $\Oph(a_\ph)$ is a smoothing operator. We have 

\begin{equation}\label{QI-Opa}
	E_\ph(a_\ph)= \langle I,  \Oph(a_\ph)(I)\rangle_\pi=I ( \Oph(a_\ph)(I))\ .	
\end{equation}
The vector $w_{a_\ph}:=\Oph(a_\ph)(I)\in V_\pi$ is a smooth vector, and we denote by $w^1_{a_\ph}\in V_\pi$ the corresponding norm one vector (assuming  $w_{a_\ph}\not=0$ since otherwise there is nothing to prove).

 We denote by ${\rm a}_\ph\geq 0$ the non-negative symbol such that $({\rm a}_\ph)^2=a_\ph$, ${\rm w}_{{\rm a}_\ph}:=\Oph({\rm a}_\ph)(I)$ and by ${\rm w}^1_{{\rm a}_\ph}$ the corresponding norm one vector. Using the composition Theorem 7.4, \cite{NV}, we see that $\Oph(a_\ph)=\Oph({\rm a}_\ph)^2+R_\ph$ with $||R_\ph||_{\pi^{-d-1}\to\pi^{\8}}\ll \ph^{1-2\dl}$.  Hence $w_{a_\ph}=\Oph(a_h)(I)=\Oph({\rm a}_\ph^2)(I)=
\Oph({\rm a}_\ph)\left[\Oph({\rm a}_\ph)(I)\right]+R_\ph(I)=\Oph({\rm a}_\ph)({\rm w}_{{\rm a}_\ph})+R_\ph(I)$.
According to \eqref{Frob-class} and the bound $tr(P_\pi|\CS_s)<C_{\pi,s}$ for $s>d$ with a constant $C_{\pi,s}>0$ depending on $s,\ \pi$, we see that $||I||_{\pi^{-s}}\leq C_{\pi,s,X}$ for $s>d$ with a constant $C_{\pi,s,X}>0$ depending on $s,\ \pi$ and (compact) $X$. Hence we have the bound $||R_\ph(I)||_\pi\ll \ph^{1-2\dl}$ and since $\dl<1/2$ this term is negligible as compared to our target bound $\ph^{-d}$. 

Let $0\not=c\in\bc$ be the proportionality constant defined by ${\rm w}_{{\rm a}_\ph}= \Oph({\rm a}_\ph)(I)=c\cdot {\rm w}_{{\rm a}_\ph}^1$. From the basic relation  $I(v)=\phi_{v}(\bar e)$,  we obtain
\begin{align*}
\label{QI-Opa-2}
	c\cdot \phi_{\Oph({\rm a}_\ph)({\rm w}_{{\rm a}_\ph}^1)}(\bar e)&=&\\ \langle I,  \Oph({\rm a}_\ph)(c\cdot {\rm w}_{{\rm a}_\ph}^1)\rangle&=&\langle \Oph({\rm a}_\ph)(I),  c\cdot {\rm w}_{{\rm a}_\ph}^1\rangle=\langle  c\cdot {\rm w}_{{\rm a}_\ph}^1,  c\cdot {\rm w}_{{\rm a}_\ph}^1\rangle=c^2\cdot || {\rm w}_{{\rm a}_\ph}^1||^2=c^2\ ,
\end{align*} and finally, $\Oph({\rm a}_\ph)(I)=\phi_{\Oph({\rm a}_\ph)({\rm w}_{{\rm a}_\ph}^1)}(\bar e)\cdot {\rm w}_{{\rm a}_\ph}^1$. 
This implies the following bound
\begin{align}	E_\ph(a_\ph)=\langle I,  \Oph(a_\ph)(I)\rangle= \langle I,  \Oph({\rm a}_\ph^2)(I)\rangle=(1+O(\ph^{1-2\dl}))\cdot\langle\Oph({\rm a}_\ph)(I),  \Oph({\rm a}_\ph)(I)\rangle=\nonumber\\ (1+O(\ph^{1-2\dl}))\cdot|\phi_{\Oph({\rm a}_\ph)({\rm w}_{{\rm a}_\ph}^1)}(\bar e)|^2\leq 2\sup_X|\phi_{\Oph({\rm a}_\ph)({\rm w}_{{\rm a}_\ph}^1)}|^2\ .\end{align}

We now need to bound sup norm of an automorphic function $|\phi_{\Oph({\rm a}_\ph)({\rm w}_{{\rm a}_\ph}^1)}|^2$ better than the dyadic bound $\ph^{-d}$. The operator  $\Oph({\rm a}_\ph)$ is essentially an idempotent and hence there is no conceptual difference between vectors  ${\Oph({\rm a}_\ph)({\rm w}_{{\rm a}_\ph}^1)}$ and  $w^1_{a_\ph}=w_{a_\ph}/||w_{a_\ph}||$ with $w_{a_\ph}:=\Oph(a_\ph)(I)$. Hence we revert to dealing with the function $\phi_{{w}_{{a}_\ph}^1}$ for notational simplicity.

The sup norm of a function could be bounded by an appropriate Sobolev norm on $X$, or more appropriately in our situation by the quantitative Frobenius reciprocity Theorem \ref{Frob-class} (note that it goes beyond the usual Sobolev restriction theorem on $X$). In particular, we have from the bound \eqref{sup-norm-bd} for $s>d$,
\begin{equation}\label{sup-norm-bd-2}
	\sup_X|\phi_{w^1_{a_\ph}}|^2\leq C_{s,X}\cdot \langle \Oph(b_\ph^{s,\kappa})({w^1_{a_\ph}}),{w^1_{a_\ph}}\rangle\ .
\end{equation} The vector ${w^1_{a_\ph}}$ is micro-localized at the $\supp(a_\ph)$, and we get (as in \cite{NV}, \S8.8.2)
\begin{equation}\label{triv-ba} |\langle \Oph(b_\ph^{s,\kappa})({w^1_{a_\ph}}),{w^1_{a_\ph}}\rangle|\leq C\sup_{\xi\in\fgh}|b_\ph^{s,\kappa}(\xi)a_\ph(\xi)|\asymp |b_\ph^{s,\kappa}(\xi_0)a_\ph(\xi_0)|\asymp \ph^{-s}\ .
\end{equation}
Hence we obtain the  bound \begin{equation}\label{triv-bd}|E_\ph(a_\ph)|\ll \ph^{-d-\eps}\end{equation} for any $\eps>0$,  since $s>d$. Such a bound is slightly  {\it weaker} than the ``trivial" bound \eqref{BR-thm-eq} valid for dyadic symbols. 

In order to improve the bound \eqref{triv-bd}, we use the fact that the sup-norm on $X$ is $G$-invariant. This will allow us to move the vector $w_{a_\ph}$ ``along $\fgh$" to decrease its $H^\ph_{s,r}$-norm and hence to improve the sup-norm bound. This in turn improves the bound on the value of $\phi_{w^1_{a_\ph}}(\bar e)=I(w^1_{a_\ph})$. We note that a similar idea was employed in \cite{BR0}, \S3.4. 

\subsection{} Let $a_\ph\in S^\8_\dl$ be as above. The equivariance of $\Oph$  amounts to the following (see \cite{NV}, \S5.5).  Let $\eps>0$. For any $g\in G$ with $||Ad(g)||\ll\ph^{-1+\dl+\eps}$, we have $\Oph(g\cdot a)\equiv \pi(g)\Oph(a)\pi(g)\inv \mod \ \ph^\8\Psi^{-\8}$ for $a\in S^\8_\dl$. Hence for $g$ as above, we have \begin{equation*}\pi(g)({ w}^1_{{ a}_\ph})=\pi(g)\left[\Oph({ a}_\ph)(I)\right]=\pi(g)\Oph({a}_\ph)\pi(g)\inv(\pi(g)(I))=\Oph(g\cdot{ a}_\ph)(\pi(g)(I))+T_\ph(I)\end{equation*} with $T_\ph\in \ph^\8\Psi^{-\8}$.  For the corresponding automorphic functions, we have 
\begin{equation}\label{action} R(g)\phi_{w^1_{a_\ph}}=R(g)\phi_{\Oph({a}_\ph)(I)}=\phi_{\pi(g)\Oph({a}_\ph)(I)}=\phi_{\Oph(g\cdot{a}_\ph)(\pi(g)I)}+\phi_{T_\ph(I)}\ ;\end{equation}
here $R(g)$ it the standard right action on functions on $X$. The operator $T_\ph:\pi^{-d-1}\to\pi^\8$ is  smoothing  with the norm bounded by any power of $\ph$. Hence the $d+1$-Sobolev norm of  $T_\ph(I)$ is negligible in $\ph$ and the same holds for the sup norm of $\phi_{T_\ph(I)}$. 

We now choose $g\in G$ based on the ``geometry" of $a_\ph$ in order to decrease the ${H^\ph_{s,r}}$-norm of $\Oph(g\cdot{a}_\ph)(\pi(g)I)$, and in turn this will improve the bound on the sup norm of $\phi_{\Oph(g\cdot{a}_\ph)(\pi(g)I)}$. Note that we assumed that $X$ is compact, and hence norms $||\pi(g)I||_{H^\ph_{s,r}}$ are uniformly bounded. Consider $\xi_0\in \CN_{reg}$ around which $a_\ph$ is supported in a $h^\dl$-ball $B_\dl=B(\xi_0,h^\dl)\subseteq \fgh$. Denote by $\theta\in (0,1)$ a (small) parameter to be optimized later.  Consider an element $g_\ph=g(\xi_0,\ph,\dl,\theta)\in G$ such that 
\begin{enumerate}
	\item Element $g_\ph$ $\ph^\theta$-contracts $\xi_0$, i.e., $||Ad^*(g_\ph)\xi_0||\ll \ph^\theta||\xi_0||$,
		\item $||Ad^*(g_\ph)||\ll \ph^{-\dl+\eps}$,
	\item $||Ad(g_\ph)||\ll \ph^{-1+\dl+\eps}$.
\end{enumerate}
Existence of (many) such $g_\ph$ is easy (e.g., by assuming that $G$ semi-simple and hence $\fg\simeq\fgh$, we can  consider Cartan subgroup in an $sl_2$-triple containing the nilpotent subalgebra $\br\xi_0\in \fgh\simeq\fg$). 
 Condition (3) allows us to use equivariance of $\Oph$ and apply $g_\ph$ to $a_\ph$ (in fact, since $\dl<1/2$ condition (2) implies  (3) and it is redundant). Conditions (1) and (2) imply that the symbol ${\rm a}^\theta_\ph:=g_\ph\cdot a_\ph$ is supported in the ball $B(Ad^*(g_\ph)\xi_0, \ph^{\dl-\theta})$. Hence for $\theta\leq \dl+\eps$, the symbol ${\rm a}^\theta_\ph$ is supported in a small ball around $\xi_0^\theta:=Ad^*(g_\ph)\xi_0$ which is much closer to $0\in\fgh$ than the original point $\xi_0$. In particular, we have a much stronger bound than \eqref{triv-ba}  following from the definition  $b_\ph^{s,\kappa}(\xi)=\ph^{-s}(\ph^{2(1-\kappa)}+|\xi|^2)^{\haf s}$ in \eqref{b-symbol} with $s=d+\eps$. Namely,  for any $\eps>0$,  
 \begin{equation}\label{non-triv-ba} \sup_{\xi\in\fgh}|b_\ph^{s,\kappa}(\xi){\rm a}^\theta_\ph|\asymp |b_\ph^{s,\kappa}(\xi_0^\theta){\rm a}^\theta_\ph(\xi_0^\theta)|\asymp \ph^{-d-\eps}(\ph^{2(1-\kappa)}+\ph^{2\theta}|\xi_0|^2)^{\haf(d+\eps)}=(\star)\  .
\end{equation} Restrictions on parameters $\kappa$, $\theta$ are those that insure that we are working with symbols allowed under $\Oph$.  We have  $\haf<\kappa<1$, $0<\theta<\dl+\eps$ with $\dl=1-\kappa<\haf$. Clearly, there is a choice of parameters such that there exists $\s>0$ such that  
 \begin{equation}\label{star-bd} (\star)\ll\ph^{-d+\s}\  .
 \end{equation}
As in \eqref{sup-norm-bd-2} and \eqref{triv-ba}, this implies the bound desired
\begin{equation}\label{non-triv-bd}|E_\ph(a_\ph)|\ll \ph^{-d+\s}\ .\end{equation}
\qed
\appendix
\section{}
\subsection{Fattening/thickening the cycle}\label{fatt} Here we sketch the proof of \cite{BR3} for  Theorem \ref{BR-thm} rephrased in the $\Oph$ language of \cite{NV}. 

The basic relation between the invariant Hermitian form $P$ on $\pi$ and the Frobenius form $Q_I$ is that   $P=\int_{g\in G/\G} g\cdot Q_Idg$. This is equivalent to requiring that the automorphic realization $\nu$ is an isometry. Hence we have the following coarse bound for any $\psi\in C_c(G)$, 
\begin{equation*}
	\int_G\psi(g) g\cdot Q_I\ dg\leq C_{\psi, \G} \cdot P\ ,
\end{equation*}with $C_{\psi,\G}=\sup_{g\in \G\sm G}|\sum_{\g\in\G}\psi( g\g)|$. 

Note that a symbol (or a part of a symbol) supported in a ball near the origin  $B(0, \ph^\bt)=\left\{\xi\in\fgh\ |\ ||\xi||\leq \ph^\bt\right\}$ for any $\haf>\bt>0$, could be bounded effectively by the form $H^\ph_{s,r}$. Hence for simplicity, we assume that the symbol $a\in C_c^\8(\fgh\sm 0)$ is compactly supported away from zero, and moreover is supported ``not far" from a non-zero point in $\CN_{reg}$. Namely, assume that there are a point $\xi_0\in\CN_{reg}$, a small enough closed ball $B_{\fgh}\subset \fgh$ such that ${\xi_0+B_{\fgh}}\cap\CN\subset\CN_{reg}$ and a ball $B_G\subset G$ satisfying $\supp(a)\subset B_G\cdot(\xi_0+B_{\fgh})$.  We want to prove that $E_\ph(\bar e)=\langle Q_I, Q_a^\ph\rangle\leq C \ph^{-d}\int_{\ph\CO_\pi} a d\om_{\ph\CO_\pi}$. We choose $\haf>\dl>0$ and consider a smooth characteristic function $a_\ph$ of the ball $\xi_0+\ph^\dl B_{\fgh}$ (e.g., $a_\ph$ as in Theorem~\ref{main-thm}). Let $\psi=\psi(a,a_\ph)\in C^\8_c(B_G)$ be a function such that $a=\int_g\psi(g)g\cdot a_\ph\  dg$. We have $Q^\ph_a=\int_G\psi(g) g \cdot Q^\ph_{a_\ph}\ dg$ and hence 
	\begin{align*}
	&E_\ph(a)=\langle Q_I,Q^\ph_a\rangle=\langle Q_I, \int_G\psi(g) g \cdot Q^\ph_{a_\ph}\ dg\rangle = \int_G\psi(g) \langle Q_I,g\cdot Q^\ph_{a_\ph}\rangle dg=\\  &\langle \int_G\psi(g)g\cdot Q_Idg, Q^\ph_{a_\ph}\rangle \leq C_{\psi,\G}\cdot\langle P, Q^\ph_{a_\ph}\rangle=C_{\psi,\G}\cdot tr(\Oph(a_\ph))\asymp C_{\psi,\G}\cdot\ph^{-d}\int_{\ph\CO_\pi} a_\ph d\om_{\ph\CO_\pi}\ .
\end{align*}	
The constant $C_{\psi,\G}$ could be interpreted as a ``packing" symbol $a$ with small balls symbols $g\cdot a_\ph$. This gives the desired result.

\subsection{Repacking}\label{repack} Here we sketch a proof of Theorem \ref{NV-thm} based on ``repacking of symbols" which is possible thanks to Ratner's equidistribution theorem. 

As in the previous section, symbols (or part thereof) supported in a small ball around origin could be bounded effectively by forms $H^\ph_{s,r}$. Hence we deal with dyadic symbols supported  ``not far" from a fixed  point $0\not=\xi_0\in\CN_{reg}$ (see above). We call such a symbol  supported at a {\it finite part} of $\CN_{reg}$. Let $a\in C^\8_c(\fgh\sm 0)$ be such a symbol. Our aim is to prove that $\lim\limits_{\ph\to 0}\ph^{d}E_\ph(a)= \int_{\ph\cal{O}_\pi}a\ d\om_{\ph\cal{O}_\pi}\ $. We will prove that for any $g\in G$, 
\begin{equation}\label{a=b}
	\lim\limits_{\ph\to 0}\ph^{d}E_\ph(g\cdot a)=\lim\limits_{\ph\to 0}\ph^{d}E_\ph(a)\ .
\end{equation} Since $\int_{g\in G/\G}g\cdot Q_I dg=P$, we obtain  on the average over $\G\sm G$ that  \begin{equation*}
\int_{g\in G/\G}E(g\cdot a)dg=\langle P,Q_a^\ph\rangle=tr(\Oph(a))=\ph^{-d}\int_{\ph\cal{O}_\pi}a\ d\om_{\ph\cal{O}_\pi}+O(\ph^{-d+1})\ . \end{equation*} This implies the desired result.

To prove \eqref{a=b}, we will prove the following ``repacking" claim.
\begin{claim*} Assume that $\G$ is co-compact. Let $a,\ b\in C^\8_c(\fgh\sm 0)$ be fixed non-negative symbols as above (i.e., supported in a finite part of $\CN_{reg}$). For any $\eps>0$, there are (finite) collections of symbols $\{a_i\}$ and $\{b_j\}$, $a_i,\  b_j\in C^\8_c(\fgh\sm 0)$ also supported at a finite part of $\CN_{reg}$, and collections of elements $\{\g^a_i\},\ \{\g^b_j\}$ in $\G$ such that 
	\begin{enumerate}
		\item $a=\sum a_i+O(\eps)$ and $b=\sum b_j+ O(\eps)$ pointwise  on $\CN_{reg}$, 
		\item $a=\frac{|a|}{|b|}\sum \g^b_jb_j+O(\eps)$ and $b=\frac{|b|}{|a|}\sum \g^a_ia_i+O(\eps)$ pointwise on $\CN_{reg}$, 
	\end{enumerate}
with $|a|=\int_{\CN_{reg}} a d\om_{\CN_{reg}}$ and the same for $b$. 
\end{claim*}

The claim means that for a co-compact $\G$, we can $\G$-repack a symbol $a$ with pieces of another symbol $b$ and vice versa, and the only $\G$-invariant of a symbol near $\CN_{reg}$ is the invariant integral along $\CN_{reg}$. The claim looks very similar to the claim that the only $\G$-invariant distribution on $\CN_{reg}$ is the invariant integral. Such a claim immediately follows from Ratner's theorem. Formally however, we need a ``finite" version of $\G$-coinvariants on functions on $\CN_{reg}$ formulated in the claim above.  

The claim above easily implies that for {\it any} $\G$-invariant Hermitian form 
$Q$ with the average bounded by $P$ (in our case $Q=Q_I$), we have $\langle Q, Q^\ph_a\rangle = \langle Q, Q^\ph_b\rangle +O(\eps\ph^{-d})$. The $O$-term comes from the fact that a pointwise bounded remainder could be bounded by a dyadic symbol with a larger support. We note that as $h\to 0$ symbols are only relevant near the nilcone $\CN$ (and in fact, the repacking is only possible near $\CN$). Since collections $\{a_i\}$ and $\{b_j\}$ are finite for any given $\eps>0$, the repacking is also valid in some neighborhood (depending on $\eps$) of $\CN$ (in fact we make this effective below). This implies \eqref{a=b} and the main claim. 

The repacking claim follows from Ratner's equidistribution theorem along the standard lines of translating equidistribution into $\G$-counting problem in certain domains in $G$.  

\subsubsection{Counting problem} Relevant to the Claim, the counting problem has the following form. Let $0\not=\xi_0\in\CN_{reg}$ and denote by $S=S_{\xi_0}=Stab_G(\xi_0)$. It is well-known that $S$ is generated by unipotent elements. Let $B=B(e,r)\subset G$ be a (small) fixed symmetric ($B=B\inv$ ) ball around $e\in G$. Let ${\rm g}\in G$ be a fixed element and $\lm=\lm_{\rm g} ={\rm g}\cdot \xi_0\in \CN_{reg}$. Consider the ball $B_\lm=B\cdot \lm\subset \CN_{reg}$ and let $B_{0}=B\cdot \xi_0\subset\CN_{reg}$. For simplicity/clarity, we prove the repacking claim for sets $B_0$ and $B_\lm$ instead of symbols $a$ and $b$. 

Let $T>1$ be a parameter and denote by $G_T:=\{g\in G\ |\ ||g||\leq T\}$. We consider the set $D_T^{\rm g}=BS{\rm g}\inv B\cap G_T$ and denote by  $\G^{\rm g}_T:=\G\cap D^{\rm g}_T$. We have $\g\in \G^{\rm g}_T$ if and only if $a_\g:=\g\cdot B_\lm\cap B_0\not=\emptyset$ for $\g\in\G$ with $||\g||\leq T$.
Our collection of $\g_i\inv$'s will be  $\G^{\rm g}_T$ (for $\eps$ calculated later) and we set as $a_i=a_{\g_i}$ the corresponding collection of ``symbols".  We need to show that $\sum_i \g_i\inv\cdot a_i$ approximates  $B_\lm$ up to any $\eps>0$ as $T\to\8$. Namely, we have to count for every $\om=b{\rm g}\xi_0$ number of $\g_i\in \G^{\rm g}_T$ such that $\g_i\cdot\om\in B_0$. All such $\g_i$ are of the form $b's{\rm g}\inv b\inv$ with $b'\in B$ and $s\in S$ arbitrary. Hence we are looking for the number of elements in $\G$ with norm bounded by $T$ and in the set $BS{\rm g}\inv b\inv$. To count these elements, we can project the set $C_T:=eBS{\rm g}\inv b\inv\cap G_T$ to $\G\sm G$ and count number of preimages of $\bar e$. This could be interpreted as an equidistribution of $S$-orbits in $X$. In our situation there is the unique $S$-invariant probability measure on $X$, i.e., the $G$-invariant measure (see \cite{NV}, Theorem 27.7). According to  Ratner's theorems, in this case ergodic averages tend to the mean  uniformly (but not effectively) as a function of  initial points.  Hence the counting is also uniform in ${\rm g}$ and  $b$,  e.g., for ${\rm g}\inv b\inv=e$, the counting number is given by $C_T(B,\bar e)=\int_{S\cap G_T}\chi_B ds$ for $\chi_B$ the characteristic function of the projection of $B$ to $X$.  We have $C_T(B,\bar e)\to vol(S\cap G_T)\cdot\frac{vol(B)}{vol(X)}$ as $T\to\8$ thanks to the Ratner's equidistribution theorem. This solves the counting problem. An effective version would imply that we can shrink the set $B=B_\ph$ as $\ph\inv\asymp T^N$ for some large but finite  $N$ depending on the quality of effective equidistribution. This would imply Conjectures (1) and (2). Note that Theorem \ref{main-thm} allows one to discard a small enough set for which Ratner's theorem possibly does not hold and substitute it with a large deviation type of statement in order to prove a variant of Conjecture (1).

\subsection*{Acknowledgments} It is a great pleasure to thank Akshay Venkatesh for pointing to a problem which led me to think about small balls and the Frobenius functional. The results presented here are part a long-term joint project with Joseph Bernstein. I thank them both for enlightening  discussions on many subjects. It is a pleasure to thank Paul Nelson for discussions on the Orbit Method, Omri Sarig on equidistribution and Vladimir Hinich on coadjoint orbits. It is a pleasure to thank Peter Sarnak for sharing challenging questions, support and interest in this line of research. 
\bibliographystyle{amsalpha}

\end{document}